\documentclass[twoside]{amsart}
\usepackage{latexsym}
\usepackage{enumitem}
\usepackage{amssymb,amsmath,amsopn}
\usepackage[dvips]{graphicx}   
\usepackage{color,epsfig}      
\usepackage{tikz} 
\usepackage{subcaption} 
\usepackage{array}
\usepackage{hyperref} \usepackage{mathtools}

\renewcommand{\Re}{\mathbb R}
\newcommand{\Red}{\Re^d}

\newcommand{\BB}{\mathbf B}
\newcommand{\B}{\mathcal B}

\newcommand{\Sph}{\mathbb{S}}

\newcommand{\Kd}{\mathcal{K}^d}

\newcommand{\Kdo}{\mathcal{K}^d_o}
\newcommand{\Ktwoo}{\mathcal{K}^2_o}

\DeclareMathOperator{\inter}{int}

\DeclareMathOperator{\conv}{conv}
\DeclareMathOperator{\proj}{proj}

\DeclareMathOperator{\vol}{vol}
\DeclareMathOperator{\bd}{bd}

\theoremstyle{plain}
\newtheorem{theorem}{Theorem}[section]
\newtheorem{lemma}[theorem]{Lemma}
\newtheorem{conjecture}{Conjecture}

\newtheorem{proposition}[theorem]{Proposition}

\newtheorem{problem}[theorem]{Problem}
\newtheorem{remark}[theorem]{Remark}
\theoremstyle{definition}
\newtheorem{definition}[theorem]{Definition}
\numberwithin{equation}{section}

\begin{document}

\title{Variants of a theorem of Macbeath in finite dimensional normed spaces}

\author[Z. L\'angi]{Zsolt L\'angi}
\author[S. Wang]{Shanshan Wang}

\address{Zsolt L\'angi, Bolyai Institute, University of Szeged,\\
Aradi v\'ertan\'uk tere 1, H-6720 Szeged, Hungary, and\\
Alfr\'ed R\'enyi Institute of Mathematics,\\
Re\'altanoda utca 13-15, H-1053, Budapest, Hungary}
\email{zlangi@server.math.u-szeged.hu}
\address{Shanshan Wang, Bolyai Institute, University of Szeged,\\
Aradi v\'ertan\'uk tere 1, H-6720 Szeged, Hungary,}
\email{shanshan\_wang87@outlook.com}

\thanks{Partially supported by the ERC Advanced Grant ``ERMiD'' and the National Research, Development and Innovation Office, NKFI, K-147544 grant, and the Project no. TKP2021-NVA-09 with the support provided by the Ministry of Innovation and Technology of Hungary from the National Research, Development and Innovation Fund and financed under the TKP2021-NVA funding scheme.}


\begin{abstract}
A classical theorem of Macbeath states that for any integers $d \geq 2$, $n \geq d+1$, $d$-dimensional Euclidean balls are hardest to approximate, in terms of volume difference, by inscribed convex polytopes with $n$ vertices. In this paper we investigate normed variants of this problem: we intend to find the extremal values of the Busemann volume, Holmes-Thompson volume, Gromov's mass and Gromov's mass$^*$ of a largest volume convex polytope with $n$ vertices, inscribed in the unit ball of a $d$-dimensional normed space.
\end{abstract}

\maketitle

\section{Introduction}\label{sec:intro}

Approximation of a convex body by convex polytopes is a long-standing problem of convex geometry (see e.g. the surveys \cite{Bar07, Bro08, Hug} or the book \cite{Gruber}). In this problem, given a convex body $K$, our goal is to find a member of a fixed family of polytopes with minimal `distance' from the given body, i.e. which minimizes a given error function depending on the body and the polytope.

In this paper we investigate a classical variant of this problem: we intend to find convex polytopes contained in the body with minimal volume difference; that is, having maximal volume. One of the first results in this regard is possibly due to Blaschke \cite[49-53]{Bl23}, who proved in 1923 that among unit area plane convex bodies, the area of a maximum area inscribed triangle is minimal for Euclidean disks. This result was generalized by Sas \cite{Sas} in 1939, proving the same statement for inscribed convex $n$-gons for any fixed integer $n \geq 3$ in place of triangles. The $d$-dimensional case of this problem with $d \geq 3$ was settled by Macbeath \cite{Macbeath}, who, in a classical paper in 1957, proved that in any dimension, Euclidean balls are hardest to approximate, in terms of volume difference, by inscribed convex polytopes with $n$ vertices for any $n \geq d+1$.
The goal of this paper is to find an analogue of this problem in finite dimensional normed spaces. We note that other variants of the theorem of Macbeath appeared in the recent manuscripts \cite{BHLL25, BGKKLRV}, and that our approach to find normal variants of the problem of Macbeath was used also in \cite{Langi16} for a different problem.

Our investigation is carried out in the $d$-dimensional real vector space $\Re^d$. During this, we use the standard notation $\inter(\cdot)$, $\bd(\cdot)$, $\conv(\cdot)$ for the interior, boundary and convex hull of a set, and denote by $[x,y]$ the closed segment with endpoints $x,y$.
We denote the family of $d$-dimensional convex bodies, and the subfamily consisting of the origin-symmetric elements, by $\Kd$ and $\Kdo$, respectively.
If $B \in \Kdo$, the normed space with $B$ as its unit ball is denoted by $\B$. The norm in this space is denoted as $|| \cdot ||_B$. We also equip $\Re^d$ with the usual inner product $\langle \cdot, \cdot \rangle$, and the induced Euclidean norm $| \cdot|$. The Euclidean unit ball in $\Re^d$, with $o$ as its center, is denoted by $\BB^d$, and we set $\Sph^{d-1}= \bd \BB^d$ and 
$\kappa_d = \lambda_d(\BB^d)$, where $\lambda_d(\cdot)$ denotes $d$-dimensional Lebesgue measure. We denote the polar of a set $S$ by $S^\circ$.
Finally, for simplicity, we call a plane convex body a \emph{disk}, and by a convex $k$-gon we mean a convex polygon with at most $k$ vertices.

Recall the well-known fact that any finite dimensional real normed space can be equipped with a Haar measure, and that it is
unique up to multiplication of the standard Lebesgue measure by a scalar.
Depending on the choice of this scalar, one may define more than one version of normed volume.
We investigate the four variants of volume that are most regularly used in the literature.
The \emph{Busemann} and \emph{Holmes-Thompson volume} of a set $S$ in a $d$-dimensional normed space with unit ball $B$, is defined as
\begin{equation}\label{eq:BusHTdef}
\vol^{Bus}_B(S) = \frac{\kappa_d}{\lambda_d(B)} \lambda_d(S) \, \hbox{and} \, \vol^{HT}_B(S) = \frac{\lambda_d(B^\circ)}{\kappa_d} \lambda_d(S),
\end{equation}
respectively.
Note that the Busemann volume of the unit ball, and the Holmes-Thompson volume of its polar, are equal to that of a Euclidean unit ball.
For \emph{Gromov's mass}, the scalar is chosen in such a way that the volume of a maximal volume cross-polytope, inscribed in the unit ball $B$
is equal to $\frac{2^d}{d!}$ , and for \emph{Gromov's mass*} (or \emph{Benson's definition of volume}), the volume of a smallest volume parallelotope, circumscribed about $B$, is equal to $2^d$.
We denote the two latter quantities by $\vol^{m}_B(S)$ and $\vol^{m*}_B(S)$, respectively. For more information on these concepts of volume, the reader is referred to \cite{PT05, PT10}.

Now we are ready to define the main problem investigated in this paper.

\begin{definition}\label{defn:normedMB}
For any $B\in\Kdo$ and $\tau\in\{Bus,HT,m,m^{*}\}$, we set
\begin{equation}\label{eq:normedMacbeath}
\mu_n^{\tau} (B) = \sup \{\vol_{B} ^{\tau} (Q): Q \hbox{ is a polytope inscribed in } B \hbox{ with at most } n \hbox{ vertices}\}.
\end{equation}
\end{definition}

\begin{problem}\label{prob:main}
For every $d \geq 2$, $n \geq d+1$ and $\tau\in\{Bus,HT,m,m^{*}\}$, find
\[
m^{\tau}(n,d)=\inf \left\{ \mu_n^{\tau}(B) : B \in \Kdo \right\}, \hbox{ and } M^{\tau}(n,d)=\sup \left\{ \mu_n^{\tau}(B) : B \in \Kdo \right\}.
\]
\end{problem}

\begin{remark}\label{rem:affineinvariance}
We observe that for any $d \geq 2$, $n \geq d+1$, $B \in \Kdo$, nondegenerate linear transformation $L : \Re^d \to \Re^d$, $n$, and $\tau\in\{Bus,HT,m,m^{*}\}$, $\mu_n^{\tau} (L(B)) = \mu_n^{\tau} (B)$. Thus, by the compactness of affine classes of $o$-symmetric convex bodies (see \cite{TJ89}), the extremal values of the functionals in Problem~\ref{prob:main} are attained at some convex bodies in $\Kdo$.
\end{remark}

In the paper, for any $B \in \Kdo$, we denote by $I(B)$ and $C(B)$ a maximum volume cross-polytope inscribed in $B$, and a minimum volume parallelotope circumscribed about $B$, respectively. Furthermore, for any $n \geq d+1$, by $Q_n(B)$ we denote a maximum volume convex polytope in $B$ with at most $n$ vertices.

\begin{remark}\label{rem:volformulas}
For any $d \geq2$, $n \geq d+1$ and $B \in \Kdo$, we have
\begin{equation}\label{eq:volformulas}
\begin{split}
\mu_n^{Bus}(B) = \frac{\lambda_d(Q_n(B))}{\lambda_d(B)} \kappa_d, \quad \mu_n^{HT}(B)= \frac{\lambda_d(Q_n(B)) \lambda_d(B^{\circ})}{\kappa_d}, \\
\mu_n^{m}(B) = \frac{\lambda_d(Q_n(B))}{\lambda_d(I(B))} \cdot \frac{2^d}{d!}, \quad \mu_n^{m*}(B) = \frac{\lambda_d(Q_n(B))}{\lambda_d(C(B))} \cdot 2^d.
\end{split}
\end{equation}
\end{remark}

The remaining part of the paper is structured as follows. In Section~\ref{sec:Radon} we give a brief introduction to Radon curves and equiframed curves, which we often use in our investigation. In Section~\ref{sec:Busemann} we present our results about the Busemann volumes of inscribed polytopes. In Sections~\ref{sec:HT}-\ref{sec:Grmstar} we investigate the Holmes-Thompson volume, Gromov's mass and Gromov's mass$^*$ of inscribed polytopes, respectively. Finally, in Section~\ref{sec:rem} we collect our additional remarks and open questions.

\section{Description of Radon curves and equiframed curves}\label{sec:Radon}

The aim of this section is to give a short introduction to Radon curves and equiframed curves. Most of the content of this section can be found in the paper \cite{MS} of Martini and Swanepoel (for more information on these topics, see also \cite{PT05, PT10} or the survey \cite{MSW}). We follow the approach in \cite{MS}.

Consider an $o$-symmetric closed, convex, continuous curve $\Gamma$ in the plane. This curve is called a \emph{Radon curve} if every boundary point is a vertex of some maximum area parallelogram inscribed in $\Gamma$. Analogously, $\Gamma$ is called an \emph{equiframed} curve, if every point of $\Gamma$ lies on the side of a smallest area parallelogram circumscribed about $\Gamma$. We note that these properties are linear invariants: for any nondegenerate linear transformation $A$, $\Gamma$ is Radon or equiframed iff $A(\Gamma)$ is Radon or equiframed, respectively. We remark also that Radon curves commonly appear in the context of normed planes: the boundary of the unit disk of a normed plane is a Radon curve iff Birkhoff orthogonality is a symmetric relation in the plane (see e.g. \cite{Langi16}).

The following statements can be found in \cite{MS}. Before stating them, we note that for any convex disk $C$ in the plane and any $p \in \bd(C)$, the union of the lines supporting $C$ at $p$ is a (possibly degenerate) double cone with apex $p$.  If $L_1, L_2$ are the lines bounding this double cone, then, with a suitable labelling of the lines $L_2$ is the image of $L_1$ under a rotation around $p$, with an acute angle measured in counterclockwise direction. In this case $L_2$ is called a \emph{right semi-tangent} of $C$ at $p$.

\begin{proposition}\label{prop:MS}
Let $C$ be an $o$-symmetric convex disk in the plane. Then the following are equivalent:
\begin{enumerate}
\item[(a)] $\bd(C)$ is an equiframed curve.
\item[(b)] Each right semi-tangent of $C$ contains a side of a circumscribed parallelogram of least area.
\item[(c)] Each point of $\bd(C)$ belongs to a side of a smallest area parallelogram circumscribed about $C$.
\end{enumerate}
Furthermore, the following are also equivalent.
\begin{enumerate}
\item[(A)] $\bd(C)$ is a Radon curve.
\item[(B)] Each supporting line of $C$ contains a side of a circumscribed parallelogram of least area.
\item[(C)] Each point of $\bd(C)$ is the midpoint of some side of a smallest area parallelogram circumscribed about $C$.
\end{enumerate}
\end{proposition}

Note that by Proposition~\ref{prop:MS}, every Radon curve is equiframed. The following construction characterizes Radon curves in the sense that every Radon curve can be obtained using this construction, and every curve constructed in this way is a Radon curve. In its description, $|x,y|$ denotes the determinant of the $2 \times 2$ matrix with columns $x,y$.

Let $\alpha > 0$, and consider vectors $e_1, e_2$ such that $|e_1,e_2| = \alpha$. Let $P$ be the $o$-symmetric parallelogram defined as $[-e_1,e_1] + [-e_2,e_2]$.
The lines containing the two generating segments of $P$ divide the plane into four (closed) quadrants, labelled I, II, III and IV. We choose the notation such that $III=-I$, $IV=-II$, $I$ contains $e_1, e_2$ on its boundary, and $e_2$ is contained also in the boundary of II.

Consider any curve $\Gamma_I$ connecting $e_1$ and $e_2$ within $P$ such that $\Gamma_I \cup [o,e_1] \cup [o,e_2]$ is the boundary of a convex disk $C_1$.
Now we define a curve $\Gamma_{II}$ in II. For each direction $e_2'$ in II different from $-e_1, e_2$, we choose $\lambda e_2'$ as the point of $\Gamma_{II}$ in this direction such that $|x,\lambda e_2'| = \alpha$ for any point $x$ of a supporting line of $C_1$, parallel to $e_2'$. Then $\Gamma_{II}$ connects $-e_1$ and $e_2$, and $\Gamma_{II} \cup [o,-e_1] \cup \cup [o,e_2]$ is a convex disk in $P \cap II$. We set $\Gamma_{III}= -\Gamma_I$ and $\Gamma_{IV}=-\Gamma_{II}$, and obtain a Radon curve $\Gamma = \Gamma_I \cup \Gamma_{II} \cup \Gamma_{III} \cup \Gamma_{IV}$.

We note that if $A$ is a linear transformation that maps $e_1, e_2$ into an orthonormal basis, then $A(\conv \Gamma)$ is an $o$-symmetric convex disk with the property that its polar coincides with the rotated copy of itself by $\frac{\pi}{2}$ around $o$.

In the last part of this section we present a construction characterizing equiframed curves. We start with a Radon curve $\Gamma$. Let $C= \conv (\Gamma)$. A \emph{wedge} of $\Gamma$ is the union $[a,b] \cup [b,c]$ of two segments in $\Gamma$, where $b$ is called the \emph{vertex} of the wedge. Consider any wedge $[a,b] \cup [b,c]$ in $\Gamma_I$. Then the corresponding part of $\Gamma_{II}$ (where the correspondence is defined by the construction of $\Gamma_{II}$ described above) is a segment $[d,e]$ in $\Gamma_{II}$ with nonregular endpoints. In particular, there are supporting lines $L_d$ and $L_e$ of $\Gamma$ at $d$ and $e$, respectively, which are parallel to $[o,a]$ and $[o,c]$. Consider any point $p$ in the triangle bounded by $[d,e]$, $L_d$ and $L_e$, and replace $[d,e]$ by $[d,p] \cup [p,e]$. We may repeat this procedure countably many times for each wedge in $\Gamma_I$. Let the modified curve obtained from $\Gamma_{II}$ in this way be denoted by $\Gamma_{II}'$, and set $\Gamma_{III}'=-\Gamma_I'$. Then $\Gamma' = \Gamma_I \cup \Gamma_{II}' \cup \Gamma_{III}' \cup \Gamma_{IV}$ is an equiframed curve, and every equiframed curve can be constructed in this way. We note also that the smallest area parallelograms circumscribed about $\Gamma$ and $\Gamma'$ have equal areas.

\section{Extremal values of the Busemann volumes of inscribed polytopes}\label{sec:Busemann}

In this section we intend to find the values of $m^{Bus}(n,d)$ and $M^{Bus}(n,d)$ for $d \geq 2$ and $n \geq d+1$. In other words, we need to find the extremal values of
\[
\mu_n^{Bus}(B) = \frac{\lambda_d(Q_n(B))}{\lambda_d(B)} \kappa_d.
\]
over all $B \in \Kdo$. We start with some simple observations.

\begin{remark}\label{rem:Bushighdimmin}
By the result of Macbeath \cite{Macbeath} mentioned in the introduction, 
for any $d \geq 2$ and $n \geq d+1$, $\frac{\lambda_d(Q_n(B))}{\lambda_d(B)}$ is minimal if $B = \BB^d$. Thus, $m^{Bus}(n,d) = \mu_n^{Bus}(\BB^d)$ for all possible values of $d$ and $n$. From the result of Sas \cite{Sas} it also follows that $m^{Bus}(n,2) = \mu_n^{Bus}(B)$ implies that $B$ is an ellipse.
\end{remark}

\begin{remark}\label{rem:Bushighdimmax}
If $n \geq 2d$, then there are $o$-symmetric polytopes in $\Re^d$ with at most $n$ vertices, implying that in this case $M^{Bus}(n,d) = \kappa_d$, and $M^{Bus}(n,d) = \mu_n^{Bus}(B)$ implies that $B$ is an $o$-symmetric polytope with at most $n$ vertices.
\end{remark}

In the remaining part of this section we investigate $M^{Bus}(n,d)$ under the condition that $d+1 \leq n \leq 2d-1$.

\begin{remark}\label{rem:Busplanemax3}
Let $T$ be a triangle containing $o$ in its midtriangle. Then an elementary computation shows that for $B= \conv (T \cup (-T))$, we have $\lambda_2(B)=2\lambda_2(T)$. On the other hand, $B$ is an $o$-symmetric convex hexagon if and only if $B= \conv (T \cup (-T))$ for some triangle $T$ with the property that $o$ is contained in its midtriangle. Thus, $M^{Bus}(3,2)=\frac{\pi}{2}$, and $\mu_3^{Bus}(B)=\frac{\pi}{2}$ if and only if $B$ is a (possibly degenerate) $o$-symmetric hexagon.
\end{remark}

Our main result in this section is Theorem~\ref{thm:Bushighdimmax}.

\begin{theorem}\label{thm:Bushighdimmax}
Let $d \geq 3$, with $d=4m+r$ for some integers $m, r$, where $0 \leq r \leq 3$. Then
\begin{itemize}
\item[(a)] $M^{Bus}(d+1,d)= \frac{\kappa_d}{\binom{d}{\lfloor \frac{d}{2}\rfloor}}$, and $M^{Bus}(d+1,d)=\mu_n^{Bus}(B)$ for $B \in \Kdo$  if $B = \conv (S \cup (-S))$ for some regular simplex centered at $o$.
\item[(b)] If $0 \leq r \leq 1$, then
\[
M^{Bus}(d+2,d)= \frac{\kappa_d}{\binom{2m+1}{m} \cdot \binom{2m+r-1}{m}}.
\]
Furthermore, $M^{Bus}(d+2,d) = \mu_n^{Bus}(B)$ for $B \in \Kdo$ if $B=\conv (S_1 \cup S_2 \cup (-S_1) \cup (-S_2))$, where $S_1$ and $S_2$ are regular simplices centered at $o$, of dimensions $2m+1$ and $2m+r-1$, respectively, lying in orthogonal linear subspaces;
\item[(c)] If $2 \leq r \leq 3$, then
\[
M^{Bus}(d+2,d)= \frac{\kappa_d}{\binom{2m+1}{m} \cdot \binom{2m+r-1}{m+1}}.
\]
Furthermore, $M^{Bus}(d+2,d) = \mu_n^{Bus}(B)$ for $B \in \Kdo$ if $B=\conv (S_1 \cup S_2 \cup (-S_1) \cup (-S_2))$, where $S_1$ and $S_2$ are regular simplices centered at $o$, of dimensions $2m+1$ and $2m+r-1$, respectively, lying in orthogonal linear subspaces.
\end{itemize}
\end{theorem}

The proof of Theorem~\ref{thm:Bushighdimmax} is based on the properties of shadow systems (also called \emph{linear parameter systems}), introduced by Rogers and Shephard \cite{RS58} and studied by many authors.

\begin{definition}\label{defn:shadowsystem}
Let $X \subset \Re^d$, let $v \in \Re^d$ be a nonzero vector, and for any $x \in X$, let $\lambda_x \in \Re$. Then the $1$-parameter family of sets
\[
C(t) = \conv \{ x + t \lambda_x v : x \in X \}, \quad t \in \Re
\]
is called a \emph{shadow system}.
\end{definition}

We use the following result of Rogers and Shepard, stated as \cite[Lemma 1]{RS58}.

\begin{lemma}\label{lem:RS}
With the notation in Definition~\ref{defn:shadowsystem}, the function $t \mapsto \lambda_d(C(t))$ is convex.
\end{lemma}

To prove Theorem~\ref{thm:Bushighdimmax}, we need some additional lemmas.

\begin{lemma}\label{lem:projections}
Let $\mathcal{H}$ be a family of finitely many hyperplanes in $\Re^d$, where the intersection of all hyperplanes is $\{ o \}$. Let $X= (x_1, x_2, \ldots, x_k) \in \left( \Re^d \right)^k$ be a finite sequence of points. Let $X_H$ denote the closure of the $k$-element point-sequences in $\left( \Re^d \right)^k$ that can be obtained from $X$ by applying finitely many subsequent projections onto hyperplanes in $\mathcal{H}$ to each element of $X$ simultaneously. Then $(o,o,\ldots, o) \in X_H$.
\end{lemma}

\begin{proof}
First, observe that $X_H$
is closed and bounded and, thus, compact. In addition, for any projection $\proj_H : \Re^d \to \Re^d$ onto a hyperplane $H \in \mathcal{H}$, we have $\proj_H(X_H) \subseteq X_H$. Indeed, if $Y=(y_1, y_2, \ldots, y_k)$ can be obtained from $X$ by finitely many orthogonal projections onto hyperplanes from $\mathcal{H}$, then the same is true about $\proj_H(Y)= (\proj_H (y_1), \proj_H(y_2), \ldots, \proj_H(y_k))$. Similarly, if $Y=(y_1, y_2, \ldots, y_k)$ can be obtained as the limit of a sequence of such point sequences $Y^m = (y_1^m, y_2^m, \ldots, y_k^m)$, then the same is true about $\proj_H(Y)$.

Let us define the function $f(Y)=\sum_{i=1}^k |y_i|$ on $\left( \Re^d \right)^k$ for any sequence $Y=(y_1,y_2, \ldots,y_k)$. This function is clearly continuous, and thus, it attains its minimum $\delta$ on $X_H$. If $\delta = 0$, then we have $(o,o,\ldots, o) \in X_H$. Suppose for contradiction that $\delta > 0$. Then, if $\sum_{i=1}^k |y_i| = \delta$  for the sequence $Y=(y_1,y_2, \ldots,y_k) \in X_H$, then, as $\bigcap\mathcal{H}=\{o\}$, there is some $H \in \mathcal{H}$ that does not contain at least one of the $y_i$. But in this case the orthogonal projection $\proj_H : \Re^d \to \Re^d$ onto $H$ strictly decreases the value of $f(Y)$, i.e. we have $f(\proj_H(Y)) < f(Y)$, a contradiction.
\end{proof}

\begin{lemma}\label{lem:whichdim}
For any integers $d \geq 3$ and $1 \leq k \leq \frac{d}{2}$, let
\[
A_d(k)= \binom{k}{\lfloor \frac{k}{2}\rfloor} \cdot \binom{d-k}{\lfloor \frac{d-k}{2}\rfloor},
\]
and let $A_d = \min \left\{ A_d(k) : 1 \leq k \leq \frac{d}{2} \right\}$.
Let $d=4m+r$ for some $r \in \{ 0,1,2,3 \}$. Then 
\begin{enumerate}
\item for $r=0$, $A_d= \binom{2m-1}{m-1} \cdot \binom{2m+1}{m}$, attained iff $k=2m-1$;
\item for $r=1$, $A_d = \binom{2m}{m} \cdot \binom{2m+1}{m}$, attained iff $k=2m-1$ or $k=2m$;
\item for $r \in \{ 2,3 \}$, $A_d = \binom{2m+1}{m} \cdot \binom{2m+r-1}{m+r-2}$, attained iff $k=2m+1$.
\end{enumerate}
\end{lemma}

\begin{proof}
Consider the case that $k \leq \frac{d}{2} -2$. Then an elementary computation yields that
\begin{equation*}
A_d(k+2)=\frac{(k+2)\cdot(k+1)\cdot\lfloor \frac{d-k}{2} \rfloor\cdot\lceil \frac{d-k}{2} \rceil}{\lfloor \frac{k+2}{2} \rfloor\cdot\lceil \frac{k+2}{2} \rceil\cdot(d-k)\cdot(d-k-1)} \cdot A_d(k).
\end{equation*}

Let  $q_d(k) = \frac{A_d(k+2)}{A_d(k)}$.
We show that $q_d(k) < 1$ for all values of $k$ and $d$. We distinguish four cases depending on the parity of $k$ and $d$. If both $k,d$ are even, then
\[
q_d(k) = \frac{(k+1)(d-k)}{(k+2)(d-k-1)}.
\]
From this, $(k+1)+(d-k)=(k+2)+(d-k-1)$ and $0 < k+1 < k+2 < d-k-1 < d-k$ implies that $q_d(k) < 1$. If $k$, $d$ or both are odd, a slight modification of this computation yields the same inequality.  

Now, we prove the assertion. If $r=0$, then, by the previous consideration, $A_d(k)$ is minimal at $k=2m-1$ or $k=2m$. On the other hand,
\[
\frac{A_d(2m)}{A_d(2m-1)} = \frac{2m+2}{2m+1} > 1,
\]
implying the statement in this case. In the cases $r=1$, $r=2$ and $r=3$ a similar computation proves the statement.
\end{proof}

\begin{proof}[Proof of Theorem~\ref{thm:Bushighdimmax}]
The proof is based on a symmetrization procedure, similar to the one in \cite{AFZ}, using shadow systems.

To prove (a), since the investigated quantity is invariant under linear transformations, we can clearly assume that $Q_{d+1}(B)$ is a translate of a regular simplex $S$ centered at the origin $o$, and $B = \conv \left( (x+S \cup (-x-S) \right)$, where $x \in \Red$.

Consider the function $f : \Red \to \Re$, $f(x) = \lambda_d(B)=\lambda_d(\conv \left( (x+S \cup (-x-S) \right))$. Moving $x$ at a constant speed at any fixed direction induces a linear parameter system on the vertices of $B$. Thus, by Lemma~\ref{lem:RS}, $f$ is a convex function. On the other hand, $f$ is invariant under any symmetry of $S$. In particular, it follows that the set of points where $f$ is minimal, is a convex set whose symmetry group contains that of $S$. Thus, $f$ is minimal at $o$. We note that the formula for the volume of $B$ in this special case can be found, e.g. in \cite{RS58_2}.

Now we deal with (b) and (c). Let $Q=Q_{d+2}(B)$. Then we may assume that $B = \conv (Q \cup (-Q))$. Using the characterization of $d$-dimensional convex polytopes with $d+2$ vertices (see e.g. \cite{Ziegler} and the references therein), for some integer $1 \leq k \leq d-1$, $Q$ can be written in the form $\conv \left( (x_1+S_1) \cup (x_2+S_2) \right)$, where
\begin{itemize}
\item[(i)] $S_1$ and $S_2$ are simplices centered at $o$, of dimensions $k$ and $d-k$, respectively;
\item[(ii)] the affine hulls of $x_1+S_1$ and $x_2+S_2$ intersect in a singleton, which is contained in both simplices.
\end{itemize}
Using the affine invariance of the problem, we may assume that $S_1$ and $S_2$ are regular simplices, and their affine hulls are orthogonal to each other. In this case the orthogonal projection of $Q$ onto the affine hull of $x_1+S_1$ is $x_1+S_1$, and the intersection of $Q$ with the affine hull of $x_2+S_2$ is $x_2+S_2$. Thus, by \cite[Lemma]{RS58_2}, the volume of $Q$ is
\[
\lambda_d(Q) = \frac{1}{\binom{d}{k}} \lambda_d(S_1) \lambda_d(S_2),
\]
independently of the position of the intersection point of $x_1+S_1$ and $x_2+S_2$ relative to $S_1$ and $S_2$. Let us define 
\[
V(x_1,x_2) = \lambda_d\left( \conv \left( (x_1+S_1) \cup (x_2+S_2) \cup (-x_1-S_1) \cup (-x_2-S_2) \right) \right).
\]
We intend to find the quantity
\[
V_{\min} (k) = \min \{ V(x_1,x_2) : x_1,x_2 \in \Re^d, (x_1+S_1) \cap (x_2+S_2) \neq \emptyset \}.
\]
Note that $V_{\min} (k) \geq \min \{ V(x_1,x_2) : x_1,x_2 \in \Re^d \}$. Thus, to find $V_{\min}(k)$ it is sufficient to show that $V(x_1,x_2)$ is minimal over $x_1,x_2 \in \Re^d$ if $x_1=x_2=o$. Let $H$ be a hyperplane of $\Re^d$ bisecting any edge of any of the simplices $S_1$ or $S_2$. Let $\proj_H : \Re^d \to H$ denote the orthogonal projection of $\Red$ onto $H$. Let $v$ denote a unit normal vector of $H$.

We show that $V(x_1,x_2) \geq V(\proj_H(x_1), \proj_H(x_2))$. To do it, we define a shadow system. Let the signed distance of $x_1$ and $x_2$ from $H$ be $\lambda_1$ and $\lambda_2$, respectively, where the sign is determined by $v$. For any $t \in \Re$ and $i \in \{ 1,2 \}$, set $x_i(t) = x_i + \lambda_i (t-1) v$, and
\[
K(t) = \conv \left( (x_1(t)+S_1) \cup (x_2(t)+S_2) \cup (-x_1(t)-S_1) \cup (-x_2(t)-S_2) \right).
\]
Then $K(t)$ is a shadow system, and thus, the function $t \mapsto \lambda_d(K(t))$ is convex by Lemma~\ref{lem:RS}. On the other hand, this function is an even function of $t$, as for any value of $t$, $K(-t)$ is the reflection of $K(t)$ to $H$. Thus, $\lambda_d(K(t))$ is minimal if $t=0$, implying, in particular, that $\lambda_d(K(1)) \geq \lambda_d(K(0))$.  But $\lambda_d(K(1)) = V(x_1,x_2)$ and $\lambda_d(K(0)) = V(\proj_H(x_1), \proj_H(x_2))$, which yields our statement.

Now, let $\mathcal{H}$ denote the family of the above bisectors of all edges of $S_1$ and $S_2$. Observe that for any $\{ i,j \} = \{ 1,2 \}$, the intersection of the bisectors of the edges of $S_i$ is the orthogonal complement of the linear hull of $S_i$, namely it is the linear hull of $S_j$. Thus, the only common point of the bisectors of the edges of both $S_1$ and $S_2$ is the origin $o$. Then, by Lemma~\ref{lem:projections}, starting with the pair $(x_1,x_2)$, the pair $(o,o)$ belongs to the closure of the set of pairs that can be obtained by finitely many projections onto elements of $\mathcal{H}$. This, our previous argument, combined with the observation that $V(x_1,x_2)$ is a continuous function of $(x_1,x_2)$, yields that $V(x_1,x_2) \geq V(o,o)$ for any $x_1,x_2 \in \Re^d$, or in other words, $V_{\min}(k) = V(o,o)$.

Next, we note that by \cite{RS58_2}, we have
\[
V(o,o)= \frac{1}{\binom{d}{k}} \lambda_d( \conv (S_1 \cup (-S_1))) \lambda_d(\conv (S_2 \cup (-S_2))).
\]
Furthermore, again by \cite{RS58_2}, we have
\[
\lambda_d( \conv (S_1 \cup (-S_1))) = \binom{k}{\left[ \frac{k}{2}\right]} \lambda_d(S_1), \hbox{ and}
\]
\[
\lambda_d( \conv (S_2 \cup (-S_2))) = \binom{d-k}{\left[ \frac{d-k}{2}\right]} \lambda_d(S_2).
\]
Combining these inequalities, we obtain that for any $o$-symmetric convex body $K$ in $\Re^d$, we have
\[
\mu_{d+2}^{Bus}(K) \leq \frac{\kappa_d}{\binom{k}{\left[ \frac{k}{2}\right]} \cdot \binom{d-k}{\left[ \frac{d-k}{2}\right]}}
\]
for some $1 \leq k \leq d-1$.
Thus the assertion in (2) and (3) of Theorem~\ref{thm:Bushighdimmax} follows from Lemma~\ref{lem:whichdim}.
\end{proof}

\section{Extremal values of the Holmes-Thompson volumes of inscribed polytopes}\label{sec:HT}

In this section, we intend to find the extremal values of
\[
\mu_n^{HT}(B)= \frac{\lambda_d(Q_n(B)) \lambda_d(B^{\circ})}{\kappa_d}.
\]
over all $B \in \Kdo$.

First, we collect our planar results.

\begin{theorem}\label{thm:HTplane}
We have the following.
\begin{itemize}
\item[(1)] For any even integer $n \geq 4$, we have $M^{HT}(n,2) = \frac{n^2}{\pi}\sin^2 \frac{\pi}{n}$, and $M^{HT}(n,2) = \mu_n^{HT}(B)$ for $B \in \Ktwoo$ if $B$ is a regular $o$-symmetric $n$-gon.
\item[(2)] We have $m^{HT}(4,2) = \frac{6}{\pi}$, and $m^{HT}(4,2) = \mu_4^{HT}(B)$ for $B \in \Ktwoo$ if $B$ is an affinely regular hexagon.
\end{itemize}
\end{theorem}

In the proof we need three lemmas. The first one, which is a generalization of a classical result of Dowker \cite{Dowker}, appeared in \cite{GFTandLFT}.

\begin{lemma}\label{lem:Dowker}
Let $K \in \Ktwoo$. Then, for every $m \geq 2$, among the maximum area $(2m)$-gons inscribed in $K$, there is an $o$-symmetric $(2m)$-gon.
\end{lemma}

For the next lemma, appearing as \cite[Theorem 1]{MR} in a paper of Meyer and Reisner, recall that for every convex body $K \subset \Re^d$, there is a unique point $s(K) \in \inter(K)$, called the \emph{Santal\'o point} of $K$, for which
the quantity $\lambda_d\left( (K-s(K))^\circ\right)$ is minimal. We note that if $K$ is $o$-symmetric, then we have $s(K)=o$.

\begin{lemma}\label{lem:MR}
Let $C(t)$ be a shadow system as defined in Definition~\ref{defn:shadowsystem}. Assume that for any $t \in [a,b]$, the set $C(t)$ has nonempty interior. Then the function
\[
t \mapsto \frac{1}{\lambda_d \left( (C(t) - s(C(t)))^\circ \right)}, t \in [a,b]
\]
is convex.
\end{lemma}

Finally, we need the following theorem of Coxeter \cite{Coxeter} (for a more general version, see \cite{Langi}).

\begin{lemma}\label{lem:Coxeter}
   Let $p_1,\dots,p_n$ be the vertices of an $n$-gon $P$ in $\mathbb{R}^2$, in cyclic order. If there is some real number $\tau\geq 0$ such that $p_{j+2}-p_{j-1}=\tau(p_{j+1}-p_{j})$ for $j = 1,\dots,n$, then $P$ is affinely regular. 
\end{lemma}


\begin{proof}[Proof of Theorem~\ref{thm:HTplane}]
First we prove (1). By Lemma~\ref{lem:Dowker}, we may assume that $Q_n(B)$ is $o$-symmetric. Since $A \subseteq B$ implies $B^{\circ} \subseteq A^\circ$, we have
$\mu_n^{HT}(B) \leq \mu_n^{HT}(Q_n(B))$. Thus, we may assume that $B=Q_n(B)$ is an $o$-symmetric $n$-gon. In this case $\mu_n^{HT}(B)$ is equal to the volume product of $B$, which is maximal if $B$ is an affinely regular convex $n$-gon (see \cite{MR2}, or \cite{AFZ} for an alternative proof).

Next, we prove (2). First, we find the minimum of $\mu_4^{HT}(B)$ in the family $\mathcal{P}_o^k$ of $o$-symmetric convex polygons with at most $k$ vertices, for an arbitrary fixed even integer $k \geq 6$. Note that by Lemma~\ref{lem:Dowker}, we may assume that $Q=Q_4(B)$ is an $o$-symmetric parallelogram in $B$. First, note that by compactness, the minimum of $\mu_4^{HT}(B)$ is attained at some $o$-symmetric polygon $\hat{B} \in \mathcal{P}_o^k$ for any value of $k$. In addition, computing the value of $\mu_4^{HT}(B)$ for a parallelogram and a regular hexagon as $B$, we see that $\hat{B}$ has at least six sides. 
We show that to find the minimum of $\mu_4^{HT}(B)$ in $\mathcal{P}_o^k$ it is sufficient to consider the affinely regular $o$-symmetric $k'$-gons for some even integer $6 \leq k' \leq k$.

Consider the case that a vertex $p$ of $\hat{B}$ is not a vertex of any maximum area parallelogram inscribed in $\hat{B}$. Then we may move $p$ and $-p$ in a direction that does not change $\lambda_2(Q_4(\hat{B}))$, but decreases $\lambda_2(\hat{B}^\circ)$. Since in this case $\mu_n^{HT}(\hat{B})$ decreases, we have a contradiction. Hence, in the remaining part we assume that every vertex of $\hat{B}$ is a vertex of at least one maximum area parallelogram inscribed in $\hat{B}$.

Let the vertices of $\hat{B}$ be $q_1, q_2, \ldots, q_{k'}$, where $k' \leq k$ is even; note that by the definition of $\mathcal{P}_o^k$ $\hat{B}$ may have stricly less vertices than $k$. Consider a maximum area parallelogram $Q$ in $\hat{B}$, with vertices $q_1, q_i, -q_1, -q_i$. Assume that $Q$ is the only maximum area parallelogram in $\hat{B}$ with $q_1$ as one of its vertices; this property is equivalent to the property that the supporting lines of $\hat{B}$ parallel to $[o,q_1]$ contain only $q_i$ and $-q_i$, respectively. Now we move $q_1$ and $-q_1$ symmetrically in a certain direction $v$. Note that by the linearity of directional derivatives, there is a unique direction $v$ such that if we move $q_1$ and $-q_1$ symmetrically in this direction, the derivative of $\lambda_2(\hat{B}^\circ)$ is zero. Note that if $v$ is not parallel to $q_i$, then we may move $q_1$ and $-q_1$ parallel to $q_i$ such that we do not change $\lambda_2(Q_4(\hat{B}))$, but decrease $\lambda_2(\hat{B}^\circ)$. Thus, $v$ is parallel to $q_i$, yielding that the line $L$ through $q_1$ parallel to $v$ supports $\hat{B}$. Similarly, if $L \cap \hat{B}$ is a nondegenerate segment, then we may move $q_1$ and $-q_1$ parallel to $L$ such that we do not change $\lambda_2(Q_4(\hat{B}))$, but we increase $\hat{B}$ with respect to inclusion, and thus, we decrease $\lambda_2(\hat{B}^\circ)$. Consequently, $\{ q_1 \} = L \cap \hat{B}$, implying that
the only maximum area parallelogram with $q_i$ as a vertex is $Q$. 

Now, let us move $q_1$ as well as $-q_1$, to points $q_1'$ and $-q_1'$ parallel to $[o,q_i]$, until we reach some other sideline of $\hat{B}$, or $[o,q_1']$ is parallel to a side of $\hat{B}$ or $[q_1',q_2]$ is parallel to $[o,q_{i+1}]$ or $[q_k,q_1']$ is parallel to $[o,q_{i-1}]$. By Lemma~\ref{lem:MR}, in this way we construct a an $o$-symmetric $k$-gon $B'$ for which $\lambda_2(Q_4(\hat{B}))=\lambda_2(Q_4(B'))$, and one of the following holds:
\begin{itemize}
\item $\mu_n^{m}(\hat{B}^\circ)) > \mu_n^{m}({B'}^\circ)$, or
\item $\mu_n^{m}(\hat{B}^\circ) = \mu_n^{m}(B'^\circ)$ and $B'$ has strictly fewer vertices, or
\item $\mu_n^{m}(\hat{B}^\circ) = \mu_n^{m}({B'}^\circ)$ and $B'$ has strictly fewer vertices $q_i$ which belong to a unique maximum area parallelogram in $B'$.
\end{itemize}
Since this process terminates after finitely many steps, we may assume that that for every vertex $q_i$ of $\hat{B}$, there is a side of $\hat{B}$ parallel to $[o,q_i]$.
By the description of Radon curves in Section~\ref{sec:Radon}, it follows that $\bd(\hat{B})$ is a Radon curve. Thus, using a suitable linear transformation on $\hat{B}$, we may assume that $\hat{B}$ is the rotated copy of $\hat{B}^\circ$ by $\frac{\pi}{2}$ around $o$. By the construction method of Radon curves described in \cite{MS}, we may assume that in this case the points $(\pm 1,0)$ and $(0, \pm 1)$ are the vertices of a largest area parallelogram in $\hat{B}$, implying also that $\lambda_2(Q_4(\hat{B}))=4$.

We show that the number $k'$ of the sides of $\hat{B}$ satisfies $k' \equiv 2 \mod 4$. Indeed, let the side parallel to $[o,q_1]$ be $[q_i,q_{i+1}]$, and observe that $q_{k'/2+1}=-q_1$. Note that then vertices $[o,q_1], \ldots, [o,q_i]$ are parallel to the sides $[q_i,q_{i+1}], \ldots, [q_{k'/2},q_{k'/2+1}]$, respectively. Thus, their numbers are equal, that is, $i=k'/2-(i-1)$. Thus, $k'/2=2i-1$, implying our statement.

In the next part we deform $\hat{B}$ so that its boundary remains a Radon curve. We carry out symmetric deformations, that is, by modifying a vertex or a sideline of $\hat{B}$ in a certain way we always assume that the opposite vertex or sideline of $\hat{B}$ is modified accordingly. Before we do it, we observe the following: the fact that $\hat{B}$ is the rotated copy of its polar by $\frac{\pi}{2}$ implies that if the point $(a,b)$ is a vertex of $\hat{B}$, then the line with equation $-bx+ay=1$ contains a side of $\hat{B}$. Based on this, we say, in general, that the point $(a,b)$ \emph{corresponds} to the line $-bx+ay=1$, and vice versa.

\begin{figure}[ht]
\begin{center}
\includegraphics[width=0.65\textwidth]{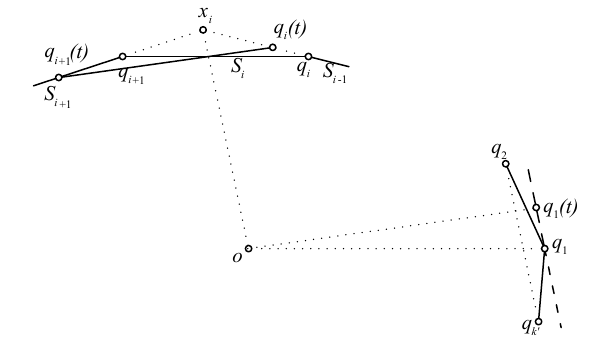}
\caption{An illustration for the modification of the Radon curve $\bd(\hat{B})$}
\label{fig:Radonmod}
\end{center}
\end{figure}

Let us denote the side parallel to $[o,q_1]$ by $S_i= [q_i,q_{i+1}]$, and let $x_i$ denote the intersection point of the sidelines through $S_{i-1}$ and $S_{i+1}$. Let us move the vertex $q_1$ continuously, on the line $L$ parallel to $[q_{k'},q_2]$ and satisfying $q_1 \in L$, at constant speed; we denote this moving vertex by $q_1(t)$, where $t$ stands for time and we set $q_1(0)=q_1$. By the properties of Radon curves, moving the point $q_1$ on $L$ corresponds to rotating the line of $S_i$ around a point $w_i$, where $w_i$ corresponds to the line $L$ in the above sense. More specifically, if we replace $\pm q_1$ by some points $\pm q_1(t)$ and simultaneously replace the sidelines of $ \pm S_i$ by some suitable lines through $\pm w_i$, then $\bd(\hat{B})$ remains a Radon curve. We also observe that, by the properties of Radon curves, these suitable lines are parallel to $[o,q_1(t)]$.

In the next step, we determine the point $w_i$. Note that since $q_1 \in L$ and $q_1$ corresponds to the line of $S_i$, $w_i$ belongs to this line. Furthermore, $L$ is parallel to $[q_{k'},q_2]$, and the line of $[q_{k'},q_2]$ corresponds to $x_i$, and thus, $w_i$ is a point of the segment $[o,x_i]$. This yields that $w_i$ is the intersection point of $[o,x_i]$ and $S_i$.

Now we define a $1$-parameter family of deformed polygons for $t \in [-\varepsilon, \varepsilon]$ for some small value of $\varepsilon$. For every such $t$ let $q_i(t)$ (resp. $q_{i+1}(t)$) denote the intersection point of $S_{i-1}$ (resp. $S_{i+1}$) with the line through $w_i$ and parallel to $[o,q_1(t)]$. Now, let $\hat{B}(t)$ be the convex $k'$-gon, obtained from $\hat{B}$ by replacing the vertices $\pm q_1, \pm q_i, \pm q_{i+1}$ by the points $\pm q_1(t), \pm q_i(t), \pm q_{i+1}(t)$, respectively. Then, for every vertex of $\hat{B}(t)$ there is a side of $\hat{B}(t)$ parallel to it, implying that $\bd(\hat{B}(t))$ is a Radon curve. For these values of $t$ we have $\lambda_2(Q_4(\hat{B}(t)))=4$.

Note that $\lambda_2 (\conv \{ q_{k'},q_1,q_2 \}) = \lambda_2 (\conv \{ q_{k'},q_1(t),q_2 \})$. Thus, if $w_i$ is not the midpoint of $[q_i,q_{i+1}]$, then for sufficiently small positive or negative values of $t$, we have $\lambda_2(\hat{B}(t)) = \lambda_2(\hat{B}^{\circ}(t)) < \lambda_2(\hat{B}) = \lambda_2(\hat{B}^{\circ})$, contradicting our assumptions. Hence, by our choice of $\hat{B}$, for every $i$ the midpoint of $[q_i,q_{i+1}]$ lies on the segment $[o,x_i]$.

In the final part we show that the above properties imply that $\hat{B}$ is an affinely regular polygon. Using a suitable linear transformation, we can assume that $o,q_{i},q_{i+1}$ are vertices of an isosceles triangle. By the properties of Radon curves, this yields that the triangle $[q_k,q_1,q_2]$ is also isosceles, and the midpoint $z_1$ of $[q_k,q_2]$ lies on $[o,q_1]$. Applying it for an arbitrary vertex we obtain that for any value of $j$, the midpoint $z_j$ of $[q_{j-1},q_{j+1}]$ belongs to $[o,q_j]$, and note that this property remains valid under any affine transformation.

\begin{figure}[ht]
\begin{center}
\includegraphics[width=0.55\textwidth]{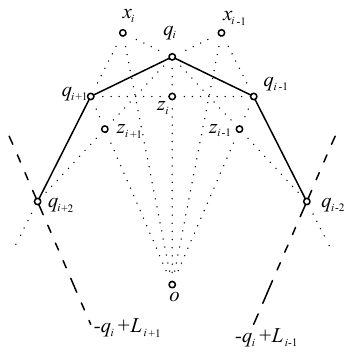}
\caption{An illustration for the proof that $\hat{B}$ is an affinely regular polygon.}
\label{fig:affreg}
\end{center}
\end{figure}

For any $i$, let $L_i$ denotes the line through $[o,q_i]$, and set $\tau_i = \frac{|q_i|}{|z_i|}$. Using a suitable linear transformation and the fact that $z_i \in L_i$, we may assume that $q_{i-1},q_{i+1}$ are symmetric to $L_i$. Observe that $x_{i-1}$ can be obtained as the intersection of the line through $S_i$, and the line through $o$ and the midpoint of $S_{i-1}$. Similarly, $x_i$ can be obtained as the intersection of the line through $S_{i-1}$, and the line through $o$ and the midpoint of $S_i$. This yields that $x_{i-1}$ and $x_i$ are symmetric to $L_i$. Furthermore, since $z_{i-1}$ is the midpoint of $[q_{i-2},q_i]$, it follows that $q_{i-2}$ lies on the line $-q_i+L_{i-1}$, and thus, it is the intersection of $-q_i+L_{i-1}$ and the line through $[x_{i-1},q_{i-1}]$. Similarly, $q_{i+2}$ is the intersection of $-q_i+L_{i+1}$ and the line through $[x_i,q_{i+1}]$. Hence, $q_{i-2}$ and $q_{i+1}$, as well as $z_{i-1}$ and $z_{i+1}$, are symmetric to $L_i$, which yields that $\tau_{i-1}=\tau_{i+1}$ (see Figure~\ref{fig:affreg}).
On the other hand, since $\hat{B}$ is $o$-symmetric and the number of its vertices is not divisible by $4$, from this it follows that there is a value $\tau$ such that $\tau=\tau_i$ for all $i$. In other words, for every value of $i$, we have
$2\tau q_i = q_{i+1}+q_{i-1}$. This yields that $2\lambda q_{i+1} = q_{i+2}+q_{i}$. Subtracting the first equality from the second one and rearranging the terms we obtain that
\[
(2\tau+1)(q_{i+1}-q_i) = q_{i+2}-q_{i-1}
\] 
for all values of $i$. Thus, by Lemma~\ref{lem:Coxeter}, $\hat{B}$ is an affinely regular polygon.

Finally, set $a_m = \mu_4^{HT}(P_{4m+2})$, where $m \geq 1$ and $P_s$ denotes a regular $s$-gon centered at $o$. An elementary computation shows that the sequence $\{ a_m \}$ is increasing, and hence, it attains its minimum at $m=1$.
\end{proof}

We note that a straightforward modification of the proof of (a) of Theorem~\ref{thm:Bushighdimmax}, replacing Lemma~\ref{lem:RS} in it with Lemma~\ref{lem:MR} readily yields the following.

\begin{theorem}
Let $d \geq 3$. Let $B= \conv(S \cup (-S))$ for some simplex $S$ centered at the origin $o$. Then $M^{HT}(d+1,d)= \mu_{d+1}^{HT}(B)$. 
\end{theorem}

\section{Extremal values of the Gromov's masses of inscribed polytopes}\label{sec:Grm}

In this section, we intend to find the extremal values of
\[
\mu_n^{m}(B) = \frac{\lambda_d(Q_n(B))}{\lambda_d(I(B))} \cdot \frac{2^d}{d!}.
\]
over all $B \in \Kdo$.

First, we collect our planar results.

\begin{theorem}\label{thm:Grmplanar}
Let $n \geq 3$. Then the following holds.
\begin{enumerate}
\item[(1)] We have $m^m(3,2)=1$, and $m^m(3,2) = \mu_3^{m}(B)$ for $B \in \Ktwoo$ if $B$ is a parallelogram.
\item[(2)] If $n \geq 4$, then $m^m(n,2) = 2$, and $m^m(n,2) = \mu_n^{m}(B)$ for $B \in \Ktwoo$ if $B$ is a parallelogram.
\item[(3)] We have $M^m(4,2) = 2$, and $M^m(4,2) = \mu_n^{m}(B)$ for any $B \in \Ktwoo$.
\item[(4)] If $n \geq 6$ is even, and $n'$ denotes the largest positive integer with $n' \leq n$ and $n' \equiv 2 \mod 4$, then  $M^m(n,2) = \frac{n' \sin \frac{2 \pi}{n'} }{2 \sin \frac{8\pi}{n'-2}}$, and $M^m(n,2) = \mu_n^{m}(B)$ if $B$ is an affinely regular $n'$-gon.
\end{enumerate}
\end{theorem}

\begin{proof}
To prove (1), note that for every $B \in \Ktwoo$,
\[
\mu_3^{m}(B) = \frac{\lambda_d(Q_3(B))}{\lambda_d(I(B))} \cdot 2 \geq \frac{\lambda_d(Q_3(I(B)))}{\lambda_d(I(B))}\cdot 2=\frac{1}{2}\cdot 2=1,
\]
with equality e.g. if $I(B)=B$.

To prove (2), observe that if $n \geq 4$, then
\[
\mu_n^{m}(B) \geq \frac{\lambda_2(I(B))}{\lambda_2(I(B))} \cdot 2 =2. 
\]
The statement in (3) follows from Lemma~\ref{lem:Dowker}. Now we prove (4).

Since $n$ is even, by Lemma~\ref{lem:Dowker}, we may assume that $Q_n(B)$ is an $o$-symmetric convex $n$-gon. Thus,
\[
\mu_n^{m}(B) = \frac{2 \lambda_2(Q_n(B))}{\lambda_2(I(B))} \leq \frac{2 \lambda_2(Q_n(B))}{\lambda_2(I(Q_n(B)))},
\]
with equality if $B=Q_n(B)$. In other words, we may assume that $B$ is an $o$-symmetric convex $n$-gon. From now on, we follow the proof of (2) of Theorem~\ref{thm:HTplane}. We only sketch the proof.

Let $\hat{B} \in \Ktwoo$ be a maximizer of the functional $\mu_n^{m}(\cdot)$, and note that $\hat{B}$ is a polygon with at most $n$ vertices.
If $\hat{B}$ has a vertex $q$ which is not a vertex of a maximum area parallelogram in $\hat{B}$, then moving $q$ and $-q$ we can increase $\lambda_2(Q_n(\hat{B})) = \lambda_2(\hat{B})$ while keeping $\lambda_2(I(\hat{B}))$ fixed, a contradiction. Thus, every vertex of $\hat{B}$ is a vertex of a largest area parallelogram in $\hat{B}$.

If $\hat{B}$ has a vertex which belongs to a unique maximal area parallelogram in $\hat{B}$,
then we can construct an $o$-symmetric $n$-gon $B'$ such that
\begin{itemize}
\item $\mu_n^{m}(\hat{B}) < \mu_n^{m}(B')$, or
\item $\mu_n^{m}(\hat{B}) = \mu_n^{m}(B')$ and $B'$ has strictly fewer vertices, or
\item $\mu_n^{m}(\hat{B}) = \mu_n^{m}(B')$ and $B'$ has strictly fewer vertices which belong to a unique maximum area parallelogram.
\end{itemize}
Thus, we may assume that every vertex $q$ of $\hat{B}$ belongs to more than one maximum area parallelogram, implying that $\bd (\hat{B})$ is a Radon curve, and that the number $n'$ of the sides of $\hat{B}$ satisfies $n' \equiv 2 \mod 4$.

Next, as in the proof of Theorem~\ref{thm:HTplane} and using the notation there, we show that if the vertices of $B_n$ are $q_1, q_2, \ldots, q_{n'}$ in counterclockwise order, $S_i = [q_i,q_{i+1}]$, and $x_i$ is the intersection point of the sidelines of $\hat{B}$ through $S_{i-1}$ and $S_{i+1}$, then the segment $[o,x_i]$ contains the midpoint of $S_i$. This, as we have seen in the proof of Theorem~\ref{thm:HTplane}, implies that $\hat{B}$ is affinely regular.

Finally, an elementary computation shows that if $g_k = \mu_{4k+2}^{m}(P_{4k+2})$, where $P_{4k+2}$ is an $o$-symmetric regular $(4k+2)$-gon, then the sequence $\{ g_k \}$ is increasing, implying (3) of Theorem~\ref{thm:Grmplanar}.
\end{proof}

\begin{remark}
We observe that a slight modification of the proof of (3) of Theorem~\ref{thm:Grmplanar} yields the following: If $n \geq 3$ is not necessarily even, then there is some $\hat{B} \in \Ktwoo$ such that $\bd(\hat{B})$ is a Radon curve, and $M^m(n,2)=\mu_n^{m}(\hat{B})$.
\end{remark}

Our higher dimensional results are collected in the next theorem. For this, we need the following straightforward consequence of Lemma~\ref{lem:RS}, the proof of which we leave to the reader.

\begin{lemma}\label{lem:vertices}
For any convex polytope $P$ in $\Red$ and $n \geq d+1$, there is a maximum volume polytope $Q_n$ with at most $n$ vertices inscribed in $P$ such that every vertex of $Q_n$ is a vertex of $P$.
\end{lemma}

\begin{theorem}\label{thm:Grmhighdim}
Let $d \geq 3$. We have the following.
\begin{enumerate}
\item[(1)] For any $n \geq 2d$, $m^m(n,d) = \frac{2^d}{d!}$. Furthermore, $m^m(n,d) = \mu_n^{m}(B)$ for $B \in \Kdo$ if $B$ is a cross-polytope.
\item[(2)] For $d+1\leq n\leq 2d-1$, $m^m(n,d) = \frac{2^d}{d!\cdot 2^{2d-n}}$, and
$m^m(n,d) = \mu_n^{m}(B)$ for $B \in \Kdo$ if $B$ is a cross-polytope.
\end{enumerate}
\end{theorem}

\begin{proof}
To prove (1), observe that $\lambda_d(Q_n(B)) \geq \lambda_d(I(B))$, with equality e.g. if $B=I(B)$. Thus,
\[
\mu_n^{m}(K) = \frac{\lambda_d(Q_n(B))}{\lambda_d(I(B))} \cdot \frac{2^d}{d!} \geq \frac{2^d}{d!},
\]
and the assertion in (1) follows.

To prove (2), note that
\[
\mu_n^{m}(B) \geq \frac{\lambda_d(Q_n(I(B)))}{\lambda_d(I(B))}\cdot \frac{2^d}{d!},
\]
and hence, we may assume that $B$ is a cross-polytope, that is $B= \conv \{ \pm q_i : i=1,2,\ldots, d\}$ for some linearly independent vectors $q_i$. By Lemma~\ref{lem:vertices}, we may assume that every vertex of $Q_n(B)$ is a vertex of $B$. Since $Q_n(B)$ is $d$-dimensional, as it has positive volume, it contains at least one element of each pair $\{q_i, -q_i \}$. Thus, the vertex set of $Q_n(B)$ consists of $(n-d)$ pairs of antipodal points, and $2d-n$ singletons. From this, an elementary computation yields the statement. 
\end{proof}

\section{Extremal values of the Gromov's mass$^*$s of inscribed polytopes}
\label{sec:Grmstar}

In this section, we intend to find the extremal values of
\[
\mu_n^{m*}(B) = \frac{\lambda_d(Q_n(B))}{\lambda_d(C(B))} \cdot 2^d.
\]
over all $B \in \Kdo$.

Our main result in this section is Theorem~\ref{thm:Grmstplanar}.

\begin{theorem}\label{thm:Grmstplanar}
We have the following.
\begin{enumerate}
\item[(1)] $M^{m*} (3,2) = 2$, and $\mu_3^{m*}(B)= M^{m*} (3,2)$ for $B \in \Ktwoo$ if $B$ is a parallelogram.
\item[(2)] For any $n \geq 4$, $M^{m*} = 4$, and $\mu_n^{m*}(B)= M^{m*} (n,2)$ for $B \in \Ktwoo$ if $B$ is a parallelogram.
\item[(3)] $m^{m*}(4,2)=2$, and $\mu_4^{m*}(B)= m^{m*} (4,2)$ for $B \in \Ktwoo$ if $\bd (B)$ is a Radon curve.
\item[(4)] If $n \geq 3$, then there is some $B \in \Ktwoo$ such that $\mu_n^{m*}(B)= m^{m*} (n,2)$ and $\bd (B)$ is a Radon curve.
\end{enumerate}
\end{theorem}

\begin{proof}
To prove (1), note that
\[
\mu_3^{m*}(B) = \frac{4\lambda_2(Q_3(B))}{\lambda_2(C(B))} \leq \frac{4\lambda_2(Q_3(C(B)))}{\lambda_2(C(B))} = 2.
\]
Similarly, to prove (2) observe that
\[
\mu_n^{m*}(B) \leq \frac{4\lambda_2(B)}{\lambda_2(B)} \leq 4,
\]
and we have equality if $B$ is a parallelogram.

For (3), we note that by Lemma~\ref{lem:Dowker}:
\[
\mu_4^{m*}(B) = \frac{4\lambda_2(I(B))}{\lambda_2(C(B))} \geq 2,
\]
and we have equality if $\bd(B)$ is a Radon curve (see \cite{MS}).

Now we prove (4). Assume that $B \in \Ktwoo$ satisfies $\mu_n^{m*}(B) = m^{m*}(n,2)$, and $B$ has minimum area among these convex disks.
We show that $\bd(B)$ is an equiframed curve; i.e. every point of $\bd(B)$ lies on a side of a minimum area parallelogram circumscribed about $B$. 
Note that for any such parallelogram $P$, the midpoints of the sides of $P$ belong to $B$. 

Suppose that there a supporting line $L$ of $B$ that is not a sideline of such a parallelogram, and note that by symmetry, $-L$ has the same property. Let $[p,q] = L \cap B$, where we possibly have $p=q$. If one of the endpoints, say $p$ is the midpoint of a side of a minimum area circumscribed parallelogram $P$, then replacing the sidelines of $P$ through $p$ and $-p$ by $L$ and $-L$ we again obtain a minimum area circumscribed parallelogram, which contradicts our assumption that $L$ is not a sideline of such a parallelogram. Thus, neither $p$ nor $q$ is the midpoint of a side of a minimum area circumscribed parallelogram. But then, slightly moving $L$ and $L'$ towards $o$ we obtain some $o$-symmetric convex disk $B' \subsetneq B$
with $\lambda_2(C(B'))=\lambda_2(C(B))$, a contradiction. This shows that $\bd(B)$ is an equiframed curve.
Thus, based an the description of equiframed curves in Section~\ref{sec:Radon}, we observe that there is some $B' \subseteq B$ in $\Ktwoo$ such that $\bd(B')$ is a Radon curve, and $\lambda_2(C(B'))=\lambda_2(C(B))$, implying the statement.
\end{proof}

\begin{remark}
Note that for any $n \geq d+1$ and $B \in \Kdo$,
\[
\mu_n^{m*}(B) = \frac{\lambda_d(Q_n(B))}{\lambda_d(C(B))} \cdot 2^d \leq 2^d.
\]
On the other hand, if $n \geq 2^d$ and $B$ a parallelotope, then $Q_n(B)=B=P(B)$, implying that in this case $M^{m*}(n,d) = 2^d$.
\end{remark}

\section{Additional remarks and open questions}\label{sec:rem}

The variant of Problem~\ref{prob:main} for Busemann volume leads to the following question:

\begin{problem}
Let $d \geq 2$ and $d+3 \leq n \leq 2d-1$. Find the minimum volume of an $o$-symmetric convex polytope in $\Re^d$ that contains a unit volume convex polytope with $n$ vertices.
\end{problem}

For our next problem, note that since for every $B \in \Ktwoo$, the sequence  
$\{ \mu_n^{HT}(B) \}$ is increasing, the sequence $\{ m^{HT}(n,2) \}$ is also increasing. On the other hand, the example of a parallelogram $P$ shows that for any $n \geq 4$,
\[
m^{HT}(n,2) \leq \frac{\lambda_2(P) \lambda_2(P^\circ)}{\pi} = \frac{8}{\pi}.
\]
Observe that for any $B \in \Ktwoo$, if $n$ is sufficiently large, then $\lambda_2(Q_n(B)) \approx \lambda_2((B))$, implying that $ \mu_n^{HT}(B) \approx \frac{\lambda_2(B) \lambda_2(B^\circ)}{\pi}$. On the other hand, a classical result of Mahler shows that $\lambda_2(B) \lambda_2(B^\circ) \geq 8$ for any $B \in \Ktwoo$, with equality for parallelograms. This suggests that if $n$ is sufficiently large, then $m^{HT}(n,2) \approx \frac{8}{\pi}$.

We conjecture the following, which implies that $m^{HT}(n,2) = \frac{8}{\pi}$ for all $n \geq 7$.

\begin{conjecture}
For any $B \in \Ktwoo$, we have
\[
\lambda_2(Q_6(B)) \lambda_2(B^{\circ}) \geq 8,
\]
with equality if $B$ is a parallelogram.
\end{conjecture}

We note that the `dual' of this problem, stating that the quantity  $\lambda_2(H(B)) \lambda_2(B^{\circ})$, where $H$ is a smallest area hexagon circumscribed about $B$, is maximal for ellipses, is equivalent to a conjecture of Makai, Jr. \cite{Makai} from 1978, which was proved very recently by Aliev \cite{Aliev}.


\section*{Acknowledgments}

We are  grateful to an anonymous referee for their careful reading and comments which helped improve this paper.

\end{document}